\documentclass[reqno, 11pt]{amsart}
\usepackage[top=4cm, bottom=3cm, left=3.2cm, right=3.2cm]{geometry}

\usepackage{amssymb,amsmath}
\usepackage{enumerate,color}

\newtheorem{thm}{Theorem}

\newtheorem{lem}{Lemma}

\allowdisplaybreaks

\newcommand{\dint}{\displaystyle\int}

\numberwithin{equation}{section} \numberwithin{lem}{section}
\numberwithin{thm}{section} \numberwithin{prop}{section}
\numberwithin{cor}{section} \numberwithin{rem}{section}

\begin{document}
\title{Uniform in time $L^{\infty}$-estimates for nonlinear aggregation-diffusion equations}
\author{Jose A. Carrillo$^{\,2}$, Jinhuan Wang$^{\,1}$}{\thanks{Corresponding author: Jinhuan Wang}}

\maketitle
\begin{center}
{\footnotesize
1-School of Mathematics, Liaoning University, Shenyang, 110036, P. R. China \\
 email: wjh800415@163.com\\
 \smallskip
2-Department of Mathematics, Imperial College London, South Kensington Campus, London SW7 2AZ, UK.
email: carrillo@imperial.ac.uk
}
\end{center}

\date{}
\begin{abstract}
We derive uniform in time $L^\infty$-bound for solutions to an aggregation-diffusion model with attractive-repulsive potentials or fully attractive potentials. We analyze two cases: either the repulsive nonlocal term dominates over the attractive part, or the diffusion term dominates over the fully attractive nonlocal part. When the repulsive part of the potential has a weaker singularity ($2-n\leq B<A\leq2$), we use the classical approach by the Sobolev and Young inequalities together with differential iterative inequalities to prove that solutions have the uniform in time $L^{\infty}$-bound. When the repulsive part of the potential has a stronger singularity ($-n<B<2-n\leq A\leq 2$), we show the uniform bounds by utilizing properties of fractional operators. We also show uniform bounds in the purely attractive case $2-n\leq A\leq 2$ within the diffusion dominated regime.
\end{abstract}

{\small {\bf Keywords:} aggregation-diffusion equations, global in time uniform estimates, Stroock-Varopoulos inequality.}

\section{Introduction}
In this paper, we analyse the Cauchy problem for the aggregation-diffusion equation
\begin{eqnarray}\label{model}
\left\{
\begin{array}{llll}
\smallskip
&\rho_t=\Delta \rho^{m} +{\rm div}(\rho\nabla (U\ast\rho)),&& x\in \mathbb R^n,~t> 0,\\
&\rho(x,0)=\rho_{0}(x), && x\in \mathbb R^n,
\end{array}
\right.
\end{eqnarray}
where $n\geq 2$, the diffusion exponent $m>1-\frac{2}{n}$, and the initial data $\rho_0(x)\in L^1_+(\mathbb R^n)\cap L^{\infty}(\mathbb R^n)$. The unknown function $\rho(x,t)$ represents the density of individuals interacting pairwise via the potential
$$
U(x)=\frac{|x|^A}{A}-\lambda\frac{|x|^B}{B}, \quad 2\geq A>B>-n\, , \, \lambda \geq0 \,.
$$
The interaction potential $U$ is attractive-repulsive for $\lambda>0$ and a fully attractive for $\lambda=0$. Here, we use the convention $\frac{|x|^0}{0}:=\log |x|$. These models arise as generalizations of Keller-Segel models for chemotaxis, and they appear as macroscopic models of collective behavior, see \cite{KeSe70,JaLu92,BlaDoPe06,BCM,BCL,ToBeLe06} and the references therein. We will restrict to two different cases in this work that depending on the singularity of the potential we classify as: weakly singular when $2-n\leq B<A\leq2$ with $\lambda\geq 0$, and strongly singular when $-n<B<2-n\leq A\leq 2$ with $\lambda>0$.

There is an extensive literature for local-in-time well-posedness results in these models when the interaction potential $U$ is chosen as the attractive Newtonian potential ($\lambda=0$, $A=2-n$), we refer to \cite{BL,BCL,Sugi1,Sugi2,Sugi3} and the references therein. Under the condition $m>1-\frac{2}{n}$, the existence of local-in-time weak solutions $\rho(x,t)\in L^\infty((0,T),L^1_+(\mathbb R^n)\cap L^{\infty}(\mathbb R^n))$ to (\ref{model}) conserving the mass, i.e.,
\begin{eqnarray}\label{M0}
\dint_{\mathbb{R}^n}\rho(x)\,dx=\dint_{\mathbb{R}^n}\rho_0(x)\,dx=:M_0\,,
\end{eqnarray}
is ensured. This result was proven in \cite{BL} only for the attractive Newtonian potential. However, the generalization of their argument to the ranges of weakly singular interaction potentials treated below is not difficult. The case of strongly singular potentials is more cumbersome but the techniques developed in \cite{CV,CVS13} should suffice. In this paper, we will assume the local-in-time well-posedness of these models, while focusing on the derivation of a-priori uniform in time estimates of the $L^p$-norms for the solutions for all $1\leq p\leq\infty$ allowing to extend the local-in-time results to global-in-time weak solutions. We will develop a full well-posedness theory for a general class of aggregation-diffusion models elsewhere.

In our present case, only stationary and equilibrium solutions have been deeply studied in \cite{CCH1,CCH2} for the fully attractive case $\lambda=0$ and the full range of exponents for $A$. The nonlinear diffusion term models very localized repulsion between individuals while the nonlocal potential takes into account non-local repulsion effects at short distances while attractive at long distances. The balance between repulsive and attractive terms determines if the density remains finite for all times uniformly or not. The exact balance in case $\lambda=0$ is given by $m=1-\tfrac{A}{n}$, called the fair-competition regime, as proven in \cite{CCH1,CCH2}. In the fair-competition regime solutions may blow-up in finite time for large enough mass $M$. This is the case for the classical Keller-Segel model ($n=2$, $m=1$, $A=0$) exhibiting the celebrated critical mass phenomena as well as for the full fair-competition regime $m=1-\tfrac{A}{n}$ with $-n<A\leq 0$. However, for the range of the parameters $m>1-\tfrac{A}{n}$, $\lambda=0$, and $-n<A\leq 0$, called the diffusion-dominated case, the existence of equilbrium solutions for all masses $M>0$ has been recently proven in \cite{CHMV}. We remind the reader that uniform in time estimates in $L^\infty$ where proven in \cite{CaCa06,Kowalczyk04} for $m>1$ for the attractive Newtonian case in two dimensions ($n=2$, $\lambda=0$, $A=0$). We also refer to \cite{CHVY} where solutions are shown to converge towards a unique equilibrium profile for diffusion-dominated regimes in 2d.

In this work, we find regimes for the general aggregation-diffusion model \eqref{model} with power-law potentials, in which there is a counterbalance due to repulsion (local or nonlocal) of the aggregation term given by the attractive $A$-term in $U$, leading to uniform $L^\infty$-bounds. The asymptotic behavior in these cases is not clarified yet, since it could be given by equilibrium solutions, see the ranges covered in \cite{CHMV}, or all the solutions may convergence to zero, quite likely in cases where the attraction is not so strong and maybe overcome by the combined effect of the nonlocal repulsion and the nonlinear diffusion term. The asymptotic behavior will also be discussed elsewhere.

The strategy of proof for all the theorems in this work can be summarized as follows. The first step is that for any fixed $p>1$, we need to obtain estimates on the $L^{p}$-norm of the solution uniformly in time with the sharpest constants possible. In the second step, we derive an important iterative differential inequality with uniform bounds in the constants involved with respect to the $L^p$-exponent. Finally, the third step is to obtain the uniform $L^{\infty}$-bound of the solution by utilizing a bootstrap iterative technique. This method follows a similar strategy to \cite{LW} but with different ingredients in each of the cases described below making them essentially different to the result in \cite{LW}. The next section treats the weak singularity case in which we take advantage of the repulsive nonlocal part to obtain the best estimates on the nonlocal term, the nonlinear diffusion term plays a less important role and it is used in order to get the regularization procedure. We also show in the purely attractive case, $\lambda=0$, that the nonlinear diffusion is able to overcome the nonlocal attraction in the diffusion-dominated case, i.e. for $m>1-\tfrac{A}{n}$. The final section deals with the strong singularity case in which we quantify the contribution of the repulsive nonlocal term by means of fractional Sobolev inequalities proved in \cite{CT}, and thus, we recover a similar result as in the weak singularity case.


\section{Uniform $L^{\infty}$-bound for the weak singularity case}
In this section, we prove a-priori uniform $L^{\infty}$-bounds for solutions to the model (\ref{model}) in the attractive-repulsive case $\lambda>0$ with $2-n\leq B<A\leq2$ and the purely attractive case $\lambda=0$ with $2-n\leq A\leq2$. For simplicity, whenever $\lambda>0$, we take $\lambda=1$ in the proofs of this section. In fact, we subdivide for $\lambda>0$ in two subcases to discuss the properties: (i) the case $2-n< B<A\leq2$, (ii) the case $2-n= B<A\leq2$. For the subcase (i), we will make use of the Sobolev inequality and the Young inequality together with interpolation inequalities (Step 1) to deduce the differential iterative relation with uniform constants in the iteration exponents (Step 2). The uniform $L^{\infty}$-bounds will be obtained as a consequence of these iterative inequalities (Step 3). In the subcase (ii), we can reduce to the subcase (i) via rearrangement inequalities. Finally for $\lambda=0$, the iterative inequalities need the additional use of the Hardy-Littlewood-Sobolev inequality.

\subsection{The case $2-n< B<A\leq2$ and $\lambda>0$}
We can compute the Laplacian of the potential that in the distributional sense gives
\begin{eqnarray}
\Delta U(x)=(A-2+n)|x|^{A-2}-(B-2+n)|x|^{B-2}=:f(|x|).\label{caseAxiao2}
\end{eqnarray}
The presence of the repulsive part of the potential allows for the following fact that strongly prevents concentrations at the origin.

\begin{lem}
If $2-n<B<A\leq2$, and $f(|x|)$ is defined in \eqref{caseAxiao2}, then there exists a zero point $r_0>0$ of $f(|x|)$ such that
\begin{eqnarray}
&& f(|x|)<0, \mbox{ for any } 0<|x|<r_0,\label{fxiao0}\\
&& f(|x|)\geq 0, \mbox{ for any } |x|\geq r_0.\label{fda0}
\end{eqnarray}
\end{lem}
\begin{proof}
From the definition of $f(|x|)$ in (\ref{caseAxiao2}), we know that
$$
f(|x|)=|x|^{B-2}\left((A-2+n)|x|^{A-B}-(B-2+n)\right).
$$
Due to $2-n<B<A\leq2$, and taking $r_0:=\left(\frac{B-2+n}{A-2+n}\right)^{\frac{1}{A-B}}>0$ satisfying $f(r_0)=0$, we can deduce
(\ref{fxiao0}) and (\ref{fda0}).
\end{proof}
\begin{thm}\label{th2.1}(Space-Time uniform $L^{\infty}$-bound)
Assume that the non-negative initial data $\rho_0\in L^{1}(\mathbb{R}^n)\cap L^{\infty}(\mathbb{R}^n)$, the diffusion exponent $m>1-\frac{2}{n}$, and $2-n<B<A\leq2$. If $\rho$ is a global weak solution, then it satisfies a uniform in time $L^\infty$-bound
$$
\|\rho\|_{L^{\infty}(\mathbb{R}_+; L^{\infty}(\mathbb{R}^n))}\leq C\,,
$$
for all $\lambda>0$.
\end{thm}
\begin{proof}
As mentioned above, we follow the three-step procedure implemented in some related degenerated Keller-Segel models in \cite{LW}.

{\bf Step 1.-} We first prove that the solution $\rho$ has uniform in time $L^p$-bounds for any fixed $p>1$. Taking $p\rho^{p-1}$ as a test function in the equation (\ref{model}), and integrating the result equation respect to $x$ in $\mathbb{R}^n$, we deduce
\begin{eqnarray}\label{timeder1}
\frac{d}{dt}\dint_{\mathbb{R}^n} \rho^{p} \,dx
=-\dfrac{4mp(p-1)}{(m+p-1)^2}\dint_{\mathbb{R}^n}|\nabla\rho^{\frac{m+p-1}{2}}|^2 dx +(p-1)\dint_{\mathbb{R}^n}\rho^{p}\Delta (U*\rho)\,dx.
\end{eqnarray}
Now, we will focus on estimating the last term in the right-hand side. From (\ref{caseAxiao2}) and (\ref{fxiao0}), we estimate it as
\begin{eqnarray}\label{rightside1}
(p-1)\dint_{\mathbb{R}^n}\rho^{p}\Delta (U*\rho)\,dx\leq (p-1)\dint_{\mathbb{R}^n}\rho^{p}(x)\int_{|x-y|\geq r_0}f(|x-y|)\rho(y)\,dy\,dx.
\end{eqnarray}
Noticing that $2-n<A\leq 2$ implies $|x-y|^{A-2}\leq r_0^{A-2}$ for $|x-y|\geq r_0$ and $B>2-n$, then from (\ref{rightside1}), we deduce that
\begin{eqnarray}\label{rightside2}
(p-1)\dint_{\mathbb{R}^n}\rho^{p}\Delta (U*\rho)\,dx\leq (p-1)(A-2+n)r_0^{A-2}M_0\dint_{\mathbb{R}^n}\rho^{p}(x)\,dx,
\end{eqnarray}
where $M_0$ is defined by (\ref{M0}).
Hence (\ref{timeder1}) and (\ref{rightside2}) gives that
\begin{eqnarray}\label{timeder2}
\frac{d}{dt}\dint_{\mathbb{R}^n} \rho^{p} \,dx
\leq-2C_1\dint_{\mathbb{R}^n}|\nabla\rho^{\frac{m+p-1}{2}}|^2 dx +C(A,B,n,M_0)(p-1)\dint_{\mathbb{R}^n}\rho^{p}(x)\,dx,
\end{eqnarray}
where $0<C_1<\tfrac{2mp(p-1)}{(m+p-1)^2}$ is a fixed constant independent of $p$.

Since $p>1$ and $m>1-\frac{2}{n}$, we have $1<p<\frac{(m+p-1)n}{n-2}$. So, the interpolation inequality implies that
\begin{eqnarray}\label{inteineq1}
\|\rho\|_{L^p}\leq\|\rho\|_{L^1}^{\theta_1}\|\rho\|^{1-\theta_1}_{L^{\frac{(m+p-1)n}{n-2}}},
\end{eqnarray}
where $\frac{1}{p}=\theta_1+\frac{n-2}{(m+p-1)n}(1-\theta_1)$, i.e.,
\begin{eqnarray*}
\theta_1=1-\frac{(p-1)(m+p-1)n}{p\big((m+p-1)n-n+2\big)},\quad 1-\theta_1=\frac{(p-1)(m+p-1)n}{p\big((m+p-1)n-n+2\big)}.
\end{eqnarray*}
Using (\ref{inteineq1}) and the Sobolev inequality, we get
\begin{eqnarray}\label{inteineq2}
\|\rho\|^p_{L^p}\leq\|\rho\|_{L^1}^{p\theta_1}\left(S_n^{-\frac{1}{2}}\|\nabla \rho^{\frac{m+p-1}{2}}\|_{L^2}\right)^{\frac{2p(1-\theta_1)}{m+p-1}},
\end{eqnarray}
where $S_n$ is the best constant of the Sobolev inequality
written as $S_n \|h\|_{L^{2^*}}^2 \leq \|\nabla h\|_{L^2}^2$ for all $h$ with finite $L^2$ gradient. Since $m>1-\frac{2}{n}$, we can directly compute that
$$
 \frac{p(1-\theta_1)}{m+p-1}= \frac{(p-1)n}{n(m+p-1)-n+2}<1,
$$
which allows us to use the Young's inequality for (\ref{inteineq2}). Thus we get
\begin{align}\label{inteineq3}
C(A,B,n,M_0)(p-1)\|\rho\|^p_{L^p}\leq& \,\sigma_1 \|\nabla \rho^{\frac{m+p-1}{2}}\|_{L^2}^{\frac{2p(1-\theta_1)}{m+p-1}q_2}\\
&+ C(\sigma_1)(C(A,B,n,M_0)(p-1))^{q_1}S_n^{-\frac{q_1p(1-\theta_1)}{m+p-1}}\|\rho\|_{L^1}^{q_1p\theta_1}\nonumber
\end{align}
for any $\sigma_1>0$ and $C(\sigma_1)=(\sigma_1 q_2)^{-\frac{q_1}{q_2}}q_1^{-1}$, where $q_1, q_2>1$ satisfies $\tfrac{1}{q_1}+\tfrac{1}{q_2}=1$.
Setting
$$
\frac{2p(1-\theta_1)}{m+p-1}q_2=2 \quad \mbox{ and } \quad \sigma_1=C_1
$$
into \eqref{inteineq3} together with (\ref{timeder2}), the differential inequality
\begin{eqnarray}\label{timeder3}
\frac{d}{dt}\dint_{\mathbb{R}^n} \rho^{p} \,dx
\leq-C_1\dint_{\mathbb{R}^n}|\nabla\rho^{\frac{m+p-1}{2}}|^2 dx +C(A,B,n,m,M_0,p)
\end{eqnarray}
holds.
Similar to (\ref{inteineq3}), we can also obtain
\begin{eqnarray}\label{inteineq4}
\|\rho\|^p_{L^p}\leq \bar\sigma_1 \|\nabla \rho^{\frac{m+p-1}{2}}\|_{L^2}^{\frac{2p(1-\theta_1)}{m+p-1}q_2}+ C(\bar\sigma_1)S_n^{-\frac{q_1p(1-\theta_1)}{m+p-1}}\|\rho\|_{L^1}^{q_1p\theta_1}
\end{eqnarray}
for any $\bar\sigma_1>0$.
Hence (\ref{timeder3}), with $\bar\sigma_1=C_1$, and (\ref{inteineq4}) results in
\begin{eqnarray*}
\frac{d}{dt}\dint_{\mathbb{R}^n} \rho^{p} \,dx
\leq -\|\rho\|^p_{L^p} +C(A,B,n,m,M_0,p).
\end{eqnarray*}
A simple computation shows that
\begin{eqnarray}\label{lpestimate1}
\|\rho\|^p_{L^p}\leq \|\rho_0\|^p_{L^p}+C(A,B,n,m,M_0,p),
\end{eqnarray}
i.e., for any fixed $p>1$, showing that the $L^p$-norm of $\rho$ is uniformly bounded in time.

\vskip 6pt

{\bf Step 2.-} We now claim that the following iterative inequality
\begin{eqnarray}\label{diedai}
\frac{d}{dt}\| \rho\|_{L^{p_k}}^{p_k}
\leq-\|\rho\|^{p_k}_{L^{p_k}} +\tilde C p_k^{\ell_2}\left(\| \rho\|_{L^{p_{k-1}}}^{p_{k-1}}\right)^{\eta}
\end{eqnarray}
holds, where $0<\eta<2$, $1<\ell_2<n+1$, $p_k=2^k+1$, $k=0,1,2,...$, and $\tilde C$ is a constant independent of $p_k$. The same procedure used to obtain (\ref{timeder2}) in Step 1 gives that
\begin{eqnarray}\label{timeder2d}
\frac{d}{dt}\|\rho\|_{L^{p_k}}^{p_k}\leq -2C_1\|\nabla\rho^{\frac{m+p_k-1}{2}}\|^2_{L^2} +C(A,B,n,M_0)p_k\|\rho\|_{L^{p_k}}^{p_k}.
\end{eqnarray}
Using again the interpolation inequality and the Sobolev inequality, we infer that	
\begin{align}\label{inteineq1d}
\|\rho\|^{p_k}_{L^{p_k}}\leq&~\|\rho\|_{L^{p_{k-1}}}^{p_k\theta_2}\|\rho\|^{(1-\theta_2)p_k}_{L^{\frac{(m+p_k-1)n}{n-2}}} \leq \|\rho\|_{L^{p_{k-1}}}^{p_k\theta_2}\left(S_n^{-\frac{1}{2}}\|\nabla \rho^{\frac{m+p_k-1}{2}}\|_{L^2}\right)^{\frac{2p_k(1-\theta_2)}{m+p_k-1}},
\end{align}
where
\begin{eqnarray}\label{theta2}
\theta_2=\frac{\frac{1}{p_k}-\frac{n-2}{n(m+p_k-1)}}{\frac{1}{p_{k-1}}-\frac{n-2}{n(m+p_k-1)}},
\quad 1-\theta_2=\frac{\frac{1}{p_{k-1}}-\frac{1}{p_k}}{\frac{1}{p_{k-1}}-\frac{n-2}{n(m+p_k-1)}}.
\end{eqnarray}
Hence, we easily compute that
$$
(1-\theta_2)\frac{p_k}{m+p_k-1}<1,\quad \mbox{ due to } m>1-\frac{2}{n},~~p_k=2^k+1.
$$
Therefore, the Young inequality implies that
\begin{align}\label{inteineq3d}
C(A,B,n,M_0)p_k\|\rho\|^{p_k}_{L^{p_k}}\leq& \,C(\sigma_1)(C(A,B,n,M_0)p_k)^{\ell_2}S_n^{-\frac{\ell_2p_k(1-\theta_2)}{m+p_k-1}}\|\rho\|_{L^{p_{k-1}}}^{\ell_2 p_k\theta_2}\nonumber \\
&+\sigma_1 \|\nabla \rho^{\frac{m+p_k-1}{2}}\|_{L^2}^{\frac{2p_k(1-\theta_2)}{m+p_k-1}\ell_1}
\end{align}
for any $\sigma_1>0$ with $C(\sigma_1)=(\sigma_1 \ell_1)^{-\frac{\ell_2}{\ell_1}}\ell_2^{-1}$, where $\ell_1, \ell_2>1$ satisfy $\tfrac{1}{\ell_1}+\tfrac{1}{\ell_2}=1$. Setting
\begin{eqnarray}\label{sigma}
\frac{2p_k(1-\theta_2)}{m+p_k-1}\ell_1=2 \quad \mbox{ and } \quad \sigma_1=C_1
\end{eqnarray}
into (\ref{inteineq3d}) together with (\ref{timeder2d}), we infer that
\begin{eqnarray}\label{timeder3d}
\frac{d}{dt}\|\rho\|_{L^{p_k}}^{p_k}\leq -C_1\|\nabla\rho^{\frac{m+p_k-1}{2}}\|^2_{L^2}  +C(\sigma_1)(C(A,B,n,M_0)p_k)^{\ell_2} S_n^{-\frac{\ell_2p_k(1-\theta_2)}{m+p_k-1}}\|\rho\|_{L^{p_{k-1}}}^{\ell_2 p_k\theta_2}.
\end{eqnarray}

On the other hand, from (\ref{inteineq1d}), we can get
\begin{eqnarray}\label{inteineq4d}
\|\rho\|^{p_k}_{L^{p_k}}\leq C(\sigma_1)S_n^{-\frac{\ell_2p_k(1-\theta_2)}{m+p_k-1}}\|\rho\|_{L^{p_{k-1}}}^{\ell_2 p_k\theta_2}+\sigma_1 \|\nabla \rho^{\frac{m+p_k-1}{2}}\|_{L^2}^{2},
\end{eqnarray}
where $\sigma_1$ and $\ell_2$ are set in (\ref{sigma}). Thus, form (\ref{timeder3d}) and (\ref{inteineq4d}) we deduce that
\begin{eqnarray}\label{timeder3d'}
\frac{d}{dt}\|\rho\|_{L^{p_k}}^{p_k}\!\leq \!-\|\rho\|_{L^{p_k}}^{p_k} \!\!+C(\sigma_1) \!\left(1\!+\!(C(A,B,n,M_0)p_k)^{\ell_2}\right) S_n^{-\frac{\ell_2p_k(1-\theta_2)}{m+p_k-1}}\!\!\!\left(\|\rho\|^{p_{k-1}}_{L^{p_{k-1}}}\right)^{\frac{\ell_2 p_k\theta_2}{p_{k-1}}}\!\!.
\end{eqnarray}
Using (\ref{theta2}) and (\ref{sigma}), and after some tedious computations, we deduce that
\begin{eqnarray}\label{ell}
\ell_2=\frac{m+p_k-1}{m+p_k-1-(1-\theta_2)p_k}<n+1,\quad \eta:=\frac{\ell_2 p_k\theta_2}{p_{k-1}}=\frac{m-1+\frac{2}{n}p_k}{m-1+\frac{2}{n}p_{k-1}}<2,
\end{eqnarray}
and
\begin{eqnarray}\label{csigma}
 C(\sigma_1)\left(1+C(A,B,n,M_0)^{\ell_2}\right) S_n^{-\frac{\ell_2p_k(1-\theta_2)}{m+p_k-1}}\to C(m,n,A,B,M_0),~~\mbox{ as } p_k\to\infty,
\end{eqnarray}
i.e., there is a constant $\tilde C:=\tilde C(m,n,A,B,M_0)>1$ independent of $p_k$ such that
$$
C(\sigma_1)\left(1+C(A,B,n,M_0)^{\ell_2}\right) S_n^{-\frac{\ell_2p_k(1-\theta_2)}{m+p_k-1}}\leq \tilde C.
$$
Hence, (\ref{timeder3d'}) implies that (\ref{diedai}) holds.

\vskip 6pt

{\bf Step 3.} In this step, we will use the iterative inequality (\ref{diedai}) to prove the $L^{\infty}$-bound of $\rho$ by induction. In (\ref{lpestimate1}) from Step 1, we take $p=p_0=2$, then it holds that
\begin{eqnarray}\label{lp0estimate}
\|\rho\|^{p_0}_{L^{p_0}}\leq \|\rho_0\|^{p_0}_{L^{p_0}}+C(A,B,n,m,M_0)\leq \bar C,
\end{eqnarray}
where $\bar C>1$ is a constant independent of $p_k$ and the time $t$.

Let $y_k(t):=\|\rho\|^{p_k}_{L^{p_k}}$, solving the iterative inequality (\ref{diedai}), we obtain
\begin{align}\label{inequa4}
\big(e^t y_k(t)\big)'\leq& ~\tilde C p_k^{\ell_2}y_{k-1}^{\eta}e^t\nonumber\\
\leq&~ \tilde C 2^{n+1}2^{(n+1)k}\max\{1,\sup_{t\geq 0}y_{k-1}^2(t)\}e^t,
\end{align}
where the last inequality used the facts $p_k=2^k+1$, $\eta<2$ and $1<\ell_2\leq n+1$.
Let $a_k:=\tilde C 2^{n+1}2^{(n+1)k}>1$ and $D_0:=\max\{1,\|\rho_0\|_{L^1},\|\rho_0\|_{L^{\infty}}\}$, then there exists $D>0$ such that $$
D_0^{\frac{p_k}{2^k}}\leq D\qquad \mbox{for all } k\,.
$$
Then
\begin{eqnarray}\label{inequa5}
y_k(0):=\|\rho_0\|^{p_k}_{L^{p_k}}\leq \left(\max\{1,\|\rho_0\|_{L^1},\|\rho_0\|_{L^{\infty}}\}\right)^{p_k}= D_0^{p_k}\leq D^{2^k}.
\end{eqnarray}
By (\ref{inequa4}) and (\ref{inequa5}), we have
\begin{align}\label{iterative}
y_k(t)\leq&~ a_k\max\{1,\sup_{t\geq 0}y_{k-1}^2(t)\}(1-e^{-t})+y_k(0)e^{-t}\nonumber\\
\leq&~ 2 a_k\max\{\sup_{t\geq 0}y_{k-1}^2(t),D^{2^k}\}.
\end{align}
From (\ref{iterative}) after some iterative steps, we have
\begin{align*}
y_k(t)\leq&~ 2a_k(2a_{k-1})^2(2a_{k-2})^{2^2}\cdot\cdot\cdot(2a_1)^{2^{k-1}}\max\{\sup_{t\geq 0}y_0^{2^k}(t),D^{2^{k}}\}\nonumber\\
=& ~(2^{n+2}\tilde C)^{2^{k}-1} (2^{n+1})^{2\cdot2^k-k-2}\max\{\sup_{t\geq 0}y_0^{2^k}(t),D^{2^{k}}\},
\end{align*}
Taking the power $1/p_k$ to above inequality, then
\begin{eqnarray}\label{yk}
\|\rho\|_{L^{p_k}}
\leq 2^{n+2} 2^{2(n+1)}\tilde C  \max\{\sup_{t\geq 0}y_0(t),D\}.
\end{eqnarray}
So, (\ref{lp0estimate}) and (\ref{yk}) imply that $\|\rho\|_{L^{\infty}}$ is uniformly bounded in time.
\end{proof}

\subsection{The case $2-n=B<A\leq 2$ and $\lambda>0$}
Notice that the potential $-\frac{|x|^{B}}{B}$ is the Newtonian potential due to $B=2-n$. Hence we can formally compute
\begin{align}
\Delta (U*\rho)(x)=&~\Delta \left(\frac{|x|^A}{A}*\rho\right)-n\alpha_n\left((-\Delta)\phi*\rho\right)\nonumber\\
=&~(A-2+n)|x|^{A-2}*\rho-n\alpha_n\rho(x).\label{caseBdeng1}
\end{align}
Here $\alpha_n$ is the volume of $n$-dimensional unit ball and $\phi(x)$ is the fundamental solution of the Laplace equation.
\begin{thm}\label{th3.1}
(Space-Time uniform $L^{\infty}$-bound)
 Assume that the initial data $\rho_0\in L^{1}_+\cap L^{\infty}(\mathbb{R}^n)$, and the diffusion exponent $m>1-\frac{2}{n}$, and $2-n=B<A\leq2$. If $\rho$ is a global weak solution, then it satisfies a uniform in time $L^\infty$-bound
$$
\|\rho\|_{L^{\infty}(\mathbb{R}_+; L^{\infty}(\mathbb{R}^n))}\leq C\,,
$$
for all $\lambda > 0$.
\end{thm}

\begin{proof}
We use the same strategy as proving Theorem \ref{th2.1}. We will be able to reduce to the the same proof as in Theorem \ref{th2.1} after the first step. We distinguish two cases.

\vskip 6pt

{\bf Case A.-} We start by considering the case $A<2$.

\vskip 6pt

{\bf Step 1.-} We first prove the uniform in time $L^p$-bound
\begin{eqnarray}\label{3lpestimate}
\|\rho\|^p_{L^p}\leq \|\rho_0\|^p_{L^p}+C(A,B,n,m,M_0,p)\quad \mbox{ for any }p>1.
\end{eqnarray}
For the case $A<2$, taking $p\rho^{p-1}$, $p>\max\{1,\frac{A-2+n}{2-A}\}$, as a test function in the equation (\ref{model}), integrating it and using (\ref{caseBdeng1}), we have
\begin{align}\label{3timeder1}
\frac{d}{dt}\dint_{\mathbb{R}^n} \rho^{p} \,dx
=&-\dfrac{4mp(p-1)}{(m+p-1)^2}\dint_{\mathbb{R}^n}|\nabla\rho^{\frac{m+p-1}{2}}|^2 dx -n\alpha_n(p-1)\dint_{\mathbb{R}^n}\rho^{p+1}\,dx\nonumber\\
&+(p-1)(A-2+n)\dint_{\mathbb{R}^n}\rho^{p}(|x|^{A-2}*\rho)\,dx.
\end{align}
Notice that for any $R>0$, it holds that
$$
|x|^{A-2}\leq |x|^{A-2}\chi_{|x|<R}(x)+R^{A-2}\,,
$$
thus the integral coming from the attracting part can be controlled as
\begin{align}\label{3timeder1new}
\dint_{\mathbb{R}^n}\rho^{p}(|x|^{A-2}*\rho)\,dx&\leq \dint_{\mathbb{R}^n}\rho^{p}((|x|^{A-2}\chi_{|x|<R}(x))*\rho)\,dx+\dint_{\mathbb{R}^n}\rho^{p}(R^{A-2}*\rho)\,dx\nonumber\\
&=: I_{R}+M_0R^{A-2}\dint_{\mathbb{R}^n}\rho^{p}\,dx.
\end{align}
Since $|x|^{A-2}$ is locally integrable for $A > 2-n$, we can get by choosing $R=R(A,n)$ small
$$
\dint_{\mathbb{R}^n}|x|^{A-2}\chi_{|x|<R}(x)\,dx\leq \frac{n\alpha_n}{A-2+n},
$$
which implies that $|x|^{A-2}\chi_{|x|<R}(x)$ is less concentrated than $\frac{n\alpha_n}{A-2+n}\delta(x)$. Hence, we make use of the Riesz rearrangement inequality to obtain that
\begin{align}\label{3timeder1new1}
I_{R}&\leq \dint_{\mathbb{R}^n}(\rho^*)^{p}((|x|^{A-2}\chi_{|x|<R}(x))*\rho^*)\,dx\nonumber \leq \frac{n\alpha_n}{A-2+n}\dint_{\mathbb{R}^n}(\rho^*)^{p+1}\,dx\nonumber\\
&=\frac{n\alpha_n}{A-2+n}\dint_{\mathbb{R}^n}\rho^{p+1}\,dx,
\end{align}
where $\rho^*$ is a rearrangement function of $\rho$.
Thus we have
\begin{align*}
(A-2+n)\dint_{\mathbb{R}^n}\rho^{p}(|x|^{A-2}*\rho)\,dx\leq n\alpha_n\dint_{\mathbb{R}^n}\rho^{p+1}\,dx +C(A,n,M_0)\dint_{\mathbb{R}^n}\rho^{p}\,dx.
\end{align*}
So, \eqref{3timeder1} becomes
\begin{align*}
\frac{d}{dt}\dint_{\mathbb{R}^n} \rho^{p} \,dx
\leq&-\dfrac{4mp(p-1)}{(m+p-1)^2}\dint_{\mathbb{R}^n}|\nabla\rho^{\frac{m+p-1}{2}}|^2 dx+(p-1)C(A,n,M_0)\dint_{\mathbb{R}^n}\rho^{p}\,dx,
\end{align*}
which is of the same form as (\ref{timeder2}) in Theorem \ref{th2.1}, and the rest of the arguments (including steps 2 and 3 in Theorem \ref{th2.1}) go through in the same way without any further change.

\vskip 6pt

{\bf Case B.-} Next, we deal with the case $A=2$.

\vskip 6pt

The formula (\ref{3timeder1}) can be written as
\begin{align*}
\frac{d}{dt}\dint_{\mathbb{R}^n} \rho^{p} \,dx
=&-\dfrac{4mp(p-1)}{(m+p-1)^2}\dint_{\mathbb{R}^n}|\nabla\rho^{\frac{m+p-1}{2}}|^2 dx -n\alpha_n(p-1)\dint_{\mathbb{R}^n}\rho^{p+1}\,dx\nonumber\\
&+n(p-1)M_0\dint_{\mathbb{R}^n}\rho^{p}\,dx\,,
\end{align*}
leading to
\begin{align*}
\frac{d}{dt}\dint_{\mathbb{R}^n} \rho^{p} \,dx
\leq&-\dfrac{4mp(p-1)}{(m+p-1)^2}\dint_{\mathbb{R}^n}|\nabla\rho^{\frac{m+p-1}{2}}|^2 dx +n(p-1)M_0\dint_{\mathbb{R}^n}\rho^{p}\,dx\,.
\end{align*}
Similar to the first case and the proof of Theorem \ref{th2.1}, an iterative inequality and the uniform $L^{\infty}$-bound of solutions can be obtained.
\end{proof}

\subsection{The case $2-n< A\leq 2$ and $\lambda =0$}

The main difference is the use of the Hardy-Littlewood-Sobolev (HLS) inequality in order to cope with the different $L^p$-norm to be bounded. Notice that the case $A=2-n$ was proven in \cite{LW}, so we reduce to the range $2-n< A\leq 2$.

\begin{thm}\label{th3.1b}
(Space-Time uniform $L^{\infty}$-bound)
Assume that the initial data $\rho_0\in L^{1}_+\cap L^{\infty}(\mathbb{R}^n)$, $2-n< A\leq 2$, and the diffusion exponent $m>1-\frac{A}{n}$ and $\lambda=0$. If $\rho$ is a global weak solution, then it satisfies a uniform in time $L^\infty$-bound
$$
\|\rho\|_{L^{\infty}(\mathbb{R}_+; L^{\infty}(\mathbb{R}^n))}\leq C\,.
$$
\end{thm}

\begin{proof} We again proceed in 3 steps.

\vskip 6pt

{\bf Step 1.-} We first prove the uniform in time $L^p$-bound
\begin{eqnarray}\label{3lpestimate}
\|\rho\|^p_{L^p}\leq \|\rho_0\|^p_{L^p}+C(A,n,m,M_0,p)\quad \mbox{ for any }p>1.
\end{eqnarray}
For the case $A<2$, taking $p\rho^{p-1}$, $p>\max\{1,\frac{A-2+n}{2-A}\}$, as a test function in the equation (\ref{model}), integrating it and using (\ref{caseBdeng1}), we have
\begin{align}\label{3timeder111}
\frac{d}{dt}\dint_{\mathbb{R}^n} \rho^{p} \,dx
=&-\dfrac{4mp(p-1)}{(m+p-1)^2}\dint_{\mathbb{R}^n}|\nabla\rho^{\frac{m+p-1}{2}}|^2 dx \nonumber\\
&+(p-1)(A-2+n)\dint_{\mathbb{R}^n}\rho^{p}(|x|^{A-2}*\rho)\,dx.
\end{align}
Using the Hardy-Littlewood-Sobolev inequality, see \cite[pp. 106]{LieLo01}, we know that
\begin{eqnarray}\label{HLSineq}
\dint_{\mathbb{R}^n\times \mathbb{R}^n}\frac{\rho^{p}(x)\rho(y)}{|x-y|^{2-A}}\,dxdy\leq C_{HLS}\|\rho^p\|_{L^r}\|\rho\|_{L^s},
\end{eqnarray}
where $\frac{1}{r}+\frac{2-A}{n}+\frac{1}{s}=2$ and the constant $C_{HLS}$ has an upper bound
 \begin{align}\label{chls11}
C_{HLS}&\leq \frac{n}{n-2+A}\left(\frac{|S^{n-1}|}{n}\right)^{\frac{2-A}{n}}\frac{1}{sr}
\left(\left(\frac{2-A}{n(1-1/s)}\right)^{\frac{2-A}{n}}+\left(\frac{2-A}{n(1-1/r)}\right)^{\frac{2-A}{n}}\right)\nonumber\\
&=:\bar C(n,A,r,s).
\end{align}
Taking $r=\frac{p+1}{p}$, then
$$
s=\frac{n(p+1)}{(n-2+A)(p+1)+n}>0.
$$
Due to $p>\max\{1,\frac{A-2+n}{2-A}\}$, we have $1<s<p+1$. Hence, we deduce
\begin{eqnarray}\label{3inteineq1}
\|\rho\|_{L^s}\leq \|\rho\|^{1-\theta}_{L^1}\|\rho\|^{\theta}_{L^{p+1}},\quad \theta=\frac{-n+(2-A)(p+1)}{np}.
\end{eqnarray}
Thus (\ref{3timeder111}), (\ref{HLSineq}) and (\ref{3inteineq1}) show that
\begin{align}\label{3timeder2}
\frac{d}{dt}\dint_{\mathbb{R}^n} \rho^{p} \,dx
\leq&-\dfrac{4mp(p-1)}{(m+p-1)^2}\dint_{\mathbb{R}^n}|\nabla\rho^{\frac{m+p-1}{2}}|^2 dx \nonumber\\
&+(p-1)(A-2+n)C_{HLS}\|\rho\|^{1-\theta}_{L^1}\|\rho\|^{p+\theta}_{L^{p+1}}.
\end{align}
Using that $A>2-n$ and the definition of $\theta$ in (\ref{3inteineq1}), we can get $1<p+\theta<p+1$. Notice that
\begin{eqnarray}\label{y22b}
\|\rho\|_{L^{p+1}}^{p+\theta}
\leq\|\rho\|_{L^{1}}^{\theta_{1}(p+\theta)}~(S_{n}^{-\frac{1}{2}}
\|\nabla\rho^{\frac{m+p-1}{2}}\|_{L^{2}})^{2(p+\theta)~\frac{1-\theta_{1}}{m+p-1}},
\end{eqnarray}
where $\theta_{1}=\frac{n(m+p-1)-(p+1)(n-2)}{(p+1)(n(m+p-1)-(n-2))}$.

Due to $A<2$ and $m>1-\frac{A}{n}$, we know $\frac{(1-\theta_{1})(p+\theta)}{m+p-1}<1$. Hence using (\ref{y22b}) and the Young inequality, we deduce that
\begin{align}\label{rightb}
C(A,n)p^2\|\rho\|^{1-\theta}_{L^1}\|\rho\|^{p+\theta}_{L^{p+1}}
\leq&~C(\sigma_{1})C(A,n)^{q_{1}}p^{2q_{1}}~\|\rho\|_{L^{1}}^{(1-\theta+\theta_{1}(p+\theta))q_{1}}~ S_{n}^{-\frac{(p+\theta)(1-\theta_{1})q_{1}}{m+p-1}}\nonumber\\
&+\sigma_{1}\|\nabla\rho^{\frac{m+p-1}{2}}\|_{L^{2}}^{\frac{2(p+\theta)(1-\theta_{1})}{m+p-1}~ q_{2}}
\end{align}
for any $\sigma_1>0$, $C(\sigma_{1})=(\sigma_{1}q_{1})^{-\frac{q_{1}}{q_{2}}}\cdot q_{1}^{-1}$, and $q_{1},q_{2}>1$ satisfy $\frac{1}{q_{1}}+\frac{1}{q_{2}}=1$.

Given $0<C_{1}<\frac{2mp(p-1)}{(m+p-1)^{2}}$, we set $\sigma_{1}=C_{1}$ and $\frac{2(p+\theta)(1-\theta_{1})q_{2}}{m+p-1}=2$, then by (\ref{3timeder2}) and (\ref{rightb}), we get
\begin{eqnarray}\label{timeder5b}
\frac{d}{dt}{\int_{\mathbb R^{n}}}\rho^{p}dx\leq-C_{1}\|\nabla\rho^{\frac{m+p-1}{2}}\|_{L^{2}}^{2}+C(A,m,n,M_{0},p).
\end{eqnarray}
Hence from (\ref{timeder5b}) and (\ref{inteineq4}), we obtain
\begin{eqnarray*}
\frac{d}{dt}{\int_{\mathbb R^{n}}}\rho^{p}dx\leq-\|\rho\|_{L^{p}}^{p}+C(A,m,n,M_{0},p).
\end{eqnarray*}
A simple computation gives
\begin{eqnarray*}
\|\rho\|_{L^{p}}^{p}\leq\|\rho_{0}\|_{L^{p}}^{p}+C(A,m,n,M_{0},p)\quad  \mbox{ for } p>\max\{1,\frac{A-2+n}{2-A}\}.
\end{eqnarray*}
Together with a simple interpolation inequality, we obtain that the uniform in time $L^p$-bound in time holds for any $p>1$.

For the case $A=2$, then (\ref{3timeder111}) can be written as
\begin{eqnarray*}
\frac{d}{dt}\dint_{\mathbb{R}^n} \rho^{p} \,dx
\leq -\dfrac{4mp(p-1)}{(m+p-1)^2}\dint_{\mathbb{R}^n}|\nabla\rho^{\frac{m+p-1}{2}}|^2 dx +n(p-1)M_0\dint_{\mathbb{R}^n}\rho^{p}\,dx.
\end{eqnarray*}
Notice that
$$
\|\rho\|^p_{L^p}\leq \|\rho\|^{p(1-\theta)}_{L^1}\|\rho\|^{p\theta}_{L^{p+1}}.
$$
Similarly to (\ref{3timeder2})-(\ref{timeder5b}), we have
\begin{eqnarray*}
\frac{d}{dt}\|\rho\|_{L^{p}}^{p}\leq -\|\rho\|_{L^{p}}^{p} +C(A,n,m,p,M_0).
\end{eqnarray*}
Solving the ordinary differential inequality, we obtain the uniform in time $L^p$-bound for the case $A=2$.
\vskip 6pt

{\bf Step 2.-} In this step, we now prove the following iterative inequality
\begin{eqnarray}\label{3diedai}
\frac{d}{dt}\| \rho\|_{L^{p_k}}^{p_k}
\leq-\|\rho\|^{p_k}_{L^{p_k}} +\tilde C p_k^{2\nu_2}\left(\left(\| \rho\|_{L^{p_{k-1}}}^{p_{k-1}}\right)^{\eta_1}+\left(\| \rho\|_{L^{p_{k-1}}}^{p_{k-1}}\right)^{\eta_2}\right),
\end{eqnarray}
where $0<\eta_1,\eta_2\leq 2$, $1<\nu_2\leq n+1$, $p_k=2^k+\frac{n}{2-A}+n$ ($k=0,1,2,...$), and $\tilde C$ is a constant independent of $p_k$.

In fact, taking $r=\frac{p_k+1}{p_k}$ and $s=\frac{n(p_k+1)}{(n-2+A)(p_k+1)+n}$ in (\ref{HLSineq}), a similar procedure to prove (\ref{3timeder2}) gives that
\begin{eqnarray}\label{3timeder2d}
\frac{d}{dt}\|\rho\|_{L^{p_k}}^{p_k}\leq -2C_1\dint_{\mathbb{R}^n}|\nabla\rho^{\frac{m+p_k-1}{2}}|^2 dx +(p_k-1)(A-2+n)C_{HLS}\|\rho\|^{1-\theta}_{L^1}\|\rho\|^{p_k+\theta}_{L^{p_k+1}},
\end{eqnarray}
where $\theta =\frac{(2-A)(p_k+1)-n}{np_k}$, $0<C_1<\frac{2mp_k(p_k-1)}{(m+p_k-1)^2}$ is a constant independent of $p_k$. Plugging $r,s$ into (\ref{chls11}), and after some tedious computations, we deduce
\begin{eqnarray*}
C_{HLS}\leq \bar C\Big(n,A,\frac{p_k+1}{p_k},\frac{n(p_k+1)}{(n-2+A)(p_k+1)+n}\Big)\leq C(n,A)(p_k+1).
\end{eqnarray*}
Hence from (\ref{3timeder2d}), we have
\begin{eqnarray*}
\frac{d}{dt}\|\rho\|_{L^{p_k}}^{p_k}\leq -2C_1\dint_{\mathbb{R}^n}|\nabla\rho^{\frac{m+p_k-1}{2}}|^2 dx +(A-2+n)p_k^2\|\rho\|^{1-\theta}_{L^1}\|\rho\|^{p_k+\theta}_{L^{p_k+1}}.
\end{eqnarray*}
Using again the interpolation inequality and the Sobolev inequality, we obtain
\begin{eqnarray*}
\|\rho\|_{p_k+1}\leq\|\rho\|_{L^{p_{k-1}}}^{\theta_1}\|\rho\|^{1-\theta_1}_{L^{\frac{(m+p_k-1)n}{n-2}}}\leq \|\rho\|_{L^{p_{k-1}}}^{\theta_1}\left(S_n^{-\frac{1}{2}}\|\nabla \rho^{\frac{m+p_k-1}{2}}\|_{L^2}\right)^{\frac{2(1-\theta_1)}{m+p_k-1}},
\end{eqnarray*}
where
\begin{eqnarray*}
\theta_1=\frac{\frac{1}{p_k+1}-\frac{n-2}{n(m+p_k-1)}}{\frac{1}{p_{k-1}}-\frac{n-2}{n(m+p_k-1)}},
\quad 1-\theta_1=\frac{\frac{1}{p_{k-1}}-\frac{1}{p_k+1}}{\frac{1}{p_{k-1}}-\frac{n-2}{n(m+p_k-1)}}.
\end{eqnarray*}
Noticing $0<p_k+\theta <p_k+1$ due to $2-n<A\leq 2$, hence we easily compute that
$$
(1-\theta_1)\frac{p_k+\theta}{m+p_k-1}<1,\quad \mbox{ due to } m>1-\frac{2}{n},~~p_k=2^k+\frac{n}{2-A}+n.
$$
Therefore, the Young inequality implies that
\begin{align}\label{3inteineq3d}
(A-2+n)p_k^2\|\rho\|^{1-\theta}_{L^1}\|\rho\|^{p_k+\theta}_{L^{p_k+1}}\leq&\, C_1(\sigma_1)((A-2+n)\|\rho\|^{1-\theta}_{L^1}p_k^2)^{\nu_2}S_n^{-\frac{\nu_2(p_k+\theta)(1-\theta_1)}{m+p_k-1}}\|\rho\|_{L^{p_{k-1}}}^{\nu_2 (p_k+\theta)\theta_1},\nonumber \\
&+\sigma_1 \|\nabla \rho^{\frac{m+p_k-1}{2}}\|_{L^2}^{\frac{2(p_k+\theta)(1-\theta_1)}{m+p_k-1}\nu_1},
\end{align}
for any $\sigma_1>0$, with $C_1(\sigma_1)=(\sigma_1 \nu_1)^{-\frac{\nu_2}{\nu_1}}\nu_2^{-1}$,
where $\nu_1, \nu_2>1$ satisfy $\tfrac{1}{\nu_1}+\tfrac{1}{\nu_2}=1$. Setting
\begin{eqnarray}\label{3sigma}
\frac{2(p_k+\theta)(1-\theta_1)}{m+p_k-1}\nu_1=2 \qquad \mbox{and} \qquad
\sigma_1=C_1\,,
\end{eqnarray}
we can compute $\nu_2=\frac{m+p_k-1}{(m+p_k-1)-(1-\theta_1)(p_k+\theta)}<n+1$.
Hence, (\ref{3timeder2d}), (\ref{3inteineq3d}) and (\ref{3sigma}) imply
\begin{align}
\frac{d}{dt}\|\rho\|_{L^{p_k}}^{p_k} \leq &\, -C_1 \dint_{\mathbb{R}^n} |\nabla\rho^{\frac{m+p_k-1}{2}}|^2 dx\nonumber \\
& + C_1(\sigma_1)(C(A,B,n,M_0)p_k^2)^{\nu_2} S_n^{-\frac{\nu_2(p_k+\theta)(1-\theta_1)}{m+p_k-1}} \left(\|\rho\|^{p_{k-1}}_{L^{p_{k-1}}}\right)^{\eta_2},
\label{3timeder3dbis}
\end{align}
where $\eta_2:=\frac{\nu_2 (p_k+\theta)\theta_1}{p_{k-1}}\leq 2$.
Thus, from (\ref{3timeder3dbis}) and (\ref{inteineq4d}) we deduce that
\begin{align}\label{3timeder4d}
\frac{d}{dt}\|\rho\|_{L^{p_k}}^{p_k}\leq& ~-\|\rho\|_{L^{p_k}}^{p_k} +C(\sigma_1) S_n^{-\frac{\ell_2p_k(1-\theta_2)}{m+p_k-1}}\left(\|\rho\|^{p_{k-1}}_{L^{p_{k-1}}}\right)^{\eta_1}\nonumber\\
&+C_1(\sigma_1)(C(A,B,n,M_0)p_k^2)^{\nu_2} S_n^{-\frac{\nu_2(p_k+\theta)(1-\theta_1)}{m+p_k-1}}\left(\|\rho\|^{p_{k-1}}_{L^{p_{k-1}}}\right)^{\eta_2},
\end{align}
where $\theta_2$, $\sigma_1$ and $\ell_2$ are same to that in (\ref{theta2}) and (\ref{sigma}), $\eta_1:=\frac{\ell_2 p_k\theta_2}{p_{k-1}}\leq2$.
Noticing that
$$
C(\sigma_1)S_n^{-\frac{\ell_2p_k(1-\theta_2)}{m+p_k-1}}\to C(m,n,A,B,M_0),~~\mbox{ as } p_k\to\infty,
$$
and
$$
C_1(\sigma_1)(C(A,B,n,M_0))^{\nu_2} S_n^{-\frac{\nu_2(p_k+\theta)(1-\theta_1)}{m+p_k-1}}\to C(m,n,A,B,M_0),~~\mbox{ as } p_k\to\infty,
$$
hence there is a constant $\tilde C:=\tilde C(m,n,A,B,M_0)>1$ independent of $p_k$ such that
$$
\max\{C(\sigma_1)S_n^{-\frac{\ell_2p_k(1-\theta_2)}{m+p_k-1}},C_1(\sigma_1)(C(A,B,n,M_0))^{\nu_2} S_n^{-\frac{\nu_2(p_k+\theta)(1-\theta_1)}{m+p_k-1}}\}\leq \tilde C.
$$
Hence, (\ref{3timeder4d}) implies that (\ref{3diedai}) holds.

\vskip 6pt

{\bf Step 3.-} In this step, we will use the iterative inequality (\ref{3diedai}) to prove the $L^{\infty}$-bound of $\rho$ by induction.
In (\ref{3lpestimate}) at Step 1, we take $p=p_0=1+\frac{n}{2-A}+n$, then it holds
\begin{eqnarray}\label{3lp0estimate}
\|\rho\|^{p_0}_{L^{p_0}}\leq \|\rho_0\|^{p_0}_{L^{p_0}}+C(A,B,n,m,M_0)\leq \bar C,
\end{eqnarray}
where $\bar C>1$ is a constant independent of $p_k$ and time $t$.

Let $y_k(t):=\|\rho\|^{p_k}_{L^{p_k}}$, solving the iterative inequality (\ref{3diedai}), we obtain
\begin{align}\label{3inequa4}
\big(e^t y_k(t)\big)'\leq&~ \tilde C p_k^{2\nu_2}\left(y_{k-1}^{\eta_2}+y_{k-1}^{\eta_2}\right)e^t\nonumber\\
\leq& ~2\tilde C \left(2n\frac{3-A}{2-A}\right)^{2(n+1)}2^{2(n+1)k}\max\{1,\sup_{t\geq 0}y_{k-1}^2(t)\}e^t,
\end{align}
where the last inequality used the facts $p_k=2^k+\frac{n}{2-A}+n\leq 2n\frac{3-A}{2-A}2^k$, $0<\eta_1,\eta_2\leq 2$ and $1<\nu_2\leq n+1$.
Let $a_k:=2\tilde C \left(2n\frac{3-A}{2-A}\right)^{2(n+1)}2^{2(n+1)k}>1$.
By (\ref{3inequa4}) and (\ref{inequa5}), we obtain
\begin{align}\label{3iterative}
y_k(t)\leq&~ a_k\max\{1,\sup_{t\geq 0}y_{k-1}^2(t)\}(1-e^{-t})+y_k(0)e^{-t}\nonumber\\
\leq& ~2 a_k\max\{1,\sup_{t\geq 0}y_{k-1}^2(t),D^{2^k}\}.
\end{align}
From (\ref{3iterative}) after some iterative steps, we have
\begin{align*}
y_k(t)\leq& ~2a_k(2a_{k-1})^2(2a_{k-2})^{2^2}\cdot\cdot\cdot(2a_1)^{2^{k-1}}\max\{\sup_{t\geq 0}y_0^{2^k}(t),D^{2^{k}}\}\nonumber\\
=& ~\left(2\tilde C \left(2n\frac{3-A}{2-A}\right)^{2(n+1)}\right)^{2^{k}-1} (2^{2(n+1)})^{2\cdot2^k-k-2}\max\{\sup_{t\geq 0}y_0^{2^k}(t),D^{2^{k}}\}.
\end{align*}
Taking the power $1/p_k$ to above inequality, then we conclude that
\begin{eqnarray}\label{3yk}
\|\rho\|_{L^{p_k}}
\leq 2\left(2n\frac{3-A}{2-A}\right)^{2(n+1)} 2^{4(n+1)}\tilde C  \max\{\sup_{t\geq 0}y_0(t),D\}.
\end{eqnarray}
So, (\ref{3lp0estimate}) and (\ref{3yk}) imply that $\|\rho\|_{L^{\infty}}$ is uniformly bounded in time.

\end{proof}


\section{$L^{\infty}$-uniform bound for the strong singularity case}

In this section, we deduce the $L^{\infty}$-uniform bound of the solutions to the model (\ref{model}) with $-n<B<2-n\leq A\leq 2$ for the attractive-repulsive $\lambda>0$. The potential $\frac{|x|^A}{A}$ has a well-defined distributional Laplacian as given in the first part of the formula \eqref{caseAxiao2}:
\begin{equation}\label{4caseAxiaodeng2}
\Delta \Big(\frac{|x|^A}{A}*\rho\Big)=\left\{
\begin{array}{ccl}
(A-2+n)|x|^{A-2}*\rho & , & \mbox{for $2-n< A\leq 2$}\\[3mm]
n\alpha_n\rho(x) & , & \mbox{for $A=2-n$}
\end{array}
\right. \,.
\end{equation}
However, the Laplacian of the repulsive part ${|x|^B}/{B}$ does not make sense for $-n<B<2-n$. Hence we utilize a fractional Laplace operator to deal with it.  Let the notation $L_s = (-\Delta)^s$ with $0 < s < 1$ for the fractional
powers of the Laplace operator define on smooth functions in $\mathbb{R}^n$ by Fourier transform
and it is extended in a natural way to functions in the Sobolev space $H^{2s} (\mathbb{R}^n)$.
The inverse operator is denoted by $\mathcal{K}_s =(-\Delta)^{-s}$ and can be realized
by convolution
\begin{eqnarray}
\mathcal{K}_s u=K_s*u,\quad K_s=C(n,s)|x|^{-(n-2s)}.\label{koperater}
\end{eqnarray}
Hence taking $s=\frac{B+n}{2}$ in (\ref{koperater}), we have
\begin{eqnarray}
\frac{|x|^{B}}{B}*u=\frac{1}{B}C\left(n,\frac{B+n}{2}\right)(-\Delta)^{-\frac{B+n}{2}} u.\label{Bopform}
\end{eqnarray}
An important tool in this case is the Stroock-Varopoulos inequality for the fractional Laplace operator $L_s = (-\Delta)^s$, $0 < s < 1$, see \cite{CV} and the references therein,
\begin{eqnarray}\label{SV}
\dint_{\mathbb{R}^n}|\rho|^{\gamma-2}\rho(-\Delta)^{\frac{\alpha}{2}} \rho\,dx\geq \frac{4(\gamma-1)}{\gamma^2}\dint_{\mathbb{R}^n}|(-\Delta)^{\frac{\alpha}{4}} |\rho|^{\frac{\gamma}{2}}|^2\,dx,
\end{eqnarray}
where $0<\alpha<2$, $\gamma>2$. We will take the advantage of the repulsive part in this section in order to compensate the attractive nonlocal part by means of the fractional Sobolev inequality \cite{CT}. As in Section 2, we will take $\lambda=1$ in the proof of the next result without loss of generality.

\begin{thm}\label{th4.1}
(Space-Time uniform $L^{\infty}$-bound)
 Assume that the initial data $\rho_0\in L^{1}_+\cap L^{\infty}(\mathbb{R}^n)$, $-n<B<2-n\leq A\leq 2$ and the diffusion exponent $m>1-\frac{2}{n}$. If $\rho$ is a global weak solution, then it satisfies a uniform in time $L^{\infty}$-bound
$$
\|\rho\|_{L^{\infty}(\mathbb{R}_+; L^{\infty}(\mathbb{R}^n))}\leq C\,,
$$
for all $\lambda> 0$.
\end{thm}

\begin{proof}
{\bf Step 1.-} We first prove a uniform in time $L^p$-bound.
Taking $p\rho^{p-1}$, $p>1$ as test function in (\ref{model}), we have
\begin{eqnarray}\label{4timeder1nn}
\frac{d}{dt}\dint_{\mathbb{R}^n} \rho^{p} \,dx
=-\dfrac{4mp(p-1)}{(m+p-1)^2}\dint_{\mathbb{R}^n}|\nabla\rho^{\frac{m+p-1}{2}}|^2 dx +(p-1)\dint_{\mathbb{R}^n}\rho^{p}\Delta (U*\rho)\,dx.
\end{eqnarray}
Now, we focus on the last term in the right-hand side of \eqref{4timeder1nn}.
For the case $2-n<A< 2$, from (\ref{4caseAxiaodeng2}) and (\ref{Bopform}), we rewrite it as
\begin{align}\label{4rightnn}
\dint_{\mathbb{R}^n}\rho^{p}\Delta (U*\rho)\,dx=&~\dint_{\mathbb{R}^n}\rho^{p}(x)(A-2+n)|x|^{A-2}*\rho(x)\,dx\nonumber\\
&+\dint_{\mathbb{R}^n}\rho^{p}(x)\Delta\left(-C\left(n,\frac{n+B}{2}\right)\frac{1}{B}(-\Delta)^{-\frac{B+n}{2}} \rho\right)\,dx\nonumber\\
=&~(A-2+n)\dint_{\mathbb{R}^n}\rho^{p}(x)|x|^{A-2}*\rho(x)\,dx\nonumber\\
&+C\left(n,\frac{n+B}{2}\right)\frac{1}{B}\dint_{\mathbb{R}^n}\rho^{p}(x)(-\Delta)^{1-\frac{B+n}{2}} \rho\,dx.
\end{align}
Using the Stroock-Varopoulos inequality (\ref{SV}) with $\gamma=p+1$, $\alpha=2-n-B$, then the repulsive term is negative
\begin{eqnarray}\label{4inequal1nn}
\frac{1}{B}\dint_{\mathbb{R}^n}\rho^{p}(x)(-\Delta)^{1-\frac{B+n}{2}} \rho\,dx\leq \frac{1}{B}\frac{4p}{(p+1)^2}\dint_{\mathbb{R}^n}|(-\Delta)^{\frac{2-n-B}{4}} \rho^{\frac{p+1}{2}}|^2\,dx< 0,
\end{eqnarray}
since $B<0$. Therefore, from (\ref{4timeder1nn}), (\ref{4rightnn}) and (\ref{4inequal1nn}), we obtain
\begin{align}\label{4timeder2nn}
\frac{d}{dt}\dint_{\mathbb{R}^n} \rho^{p} \,dx
\leq&-\dfrac{4mp(p-1)}{(m+p-1)^2}\dint_{\mathbb{R}^n}|\nabla\rho^{\frac{m+p-1}{2}}|^2 dx \nonumber\\ &+(p-1)(A-2+n)\dint_{\mathbb{R}^n}\dint_{\mathbb{R}^n}\frac{\rho^{p}(x)\rho(y)}{|x-y|^{2-A}}\,dydx\nonumber\\
&+C\left(n,\frac{n+B}{2}\right)\frac{1}{B}\frac{4p(p-1)}{(p+1)^2}\dint_{\mathbb{R}^n}|(-\Delta)^{\frac{2-n-B}{4}} \rho^{\frac{p+1}{2}}|^2\,dx.
\end{align}
Using the same method as (\ref{3timeder1new})-(\ref{3timeder1new1}), we can still obtain that
$$
(A-2+n)\dint_{\mathbb{R}^n}\dint_{\mathbb{R}^n}\frac{\rho^{p}(x)\rho(y)}{|x-y|^{2-A}}\,dydx\leq n\alpha_n\dint_{\mathbb{R}^n}\rho^{p+1}\,dx+C(M_0,A,n)\dint_{\mathbb{R}^n}\rho^p\,dx.
$$
Hence (\ref{4timeder2nn}) can be written as
\begin{align}\label{4timeder2nnn}
\frac{d}{dt}\dint_{\mathbb{R}^n} \rho^{p} \,dx
\leq&-\dfrac{4mp(p-1)}{(m+p-1)^2}\dint_{\mathbb{R}^n}|\nabla\rho^{\frac{m+p-1}{2}}|^2 dx \nonumber\\ &+(p-1)n\alpha_n\dint_{\mathbb{R}^n}\rho^{p+1}\,dx+(p-1)C(M_0,A,n)\dint_{\mathbb{R}^n}\rho^p\,dx\nonumber\\
&+C\left(n,\frac{n+B}{2}\right)\frac{1}{B}\frac{4p(p-1)}{(p+1)^2}\dint_{\mathbb{R}^n}|(-\Delta)^{\frac{2-n-B}{4}} \rho^{\frac{p+1}{2}}|^2\,dx.
\end{align}
Using the interpolation inequality and the fractional Sobolev inequality \cite{CT}, we deduce that
\begin{align}\label{p+1}
\|\rho\|^{p+1}_{L^{p+1}}&\leq M_0^{(1-\theta)(p+1)}\|\rho\|^{(p+1)\theta}_{L^{\frac{n(p+1)}{2n-2+B}}} = M_0^{(1-\theta)(p+1)}\|\rho^{\frac{p+1}{2}}\|^{2\theta}_{L^{\frac{2n}{2n-2+B}}}\nonumber\\
&\leq S\left(n,\frac{2-n-B}{2}\right)^{\theta} M_0^{(1-\theta)(p+1)}\|(-\Delta)^{\frac{2-n-B}{4}}\rho^{\frac{p+1}{2}}\|^{2\theta}_{L^{2}},
\end{align}
where $\theta=\frac{pn}{pn+2-n-B}$, and $S(n,s)$ is given by the following exact form
$$
S(n,s)=2^{-2s}\pi^{-s}\frac{\Gamma(\frac{n-2s}{2})}{\Gamma(\frac{n+2s}{2})}\left[\frac{\Gamma(n)}{\Gamma(\frac{n}{2})}\right]^{2s/n}.
$$

Making use of the Young inequality and the formula (\ref{4timeder2nnn})-(\ref{p+1}), we infer that for any $\sigma_0>0$, it holds that
\begin{align*}
\frac{d}{dt}\dint_{\mathbb{R}^n} \rho^{p} \,dx
\leq&-\dfrac{4mp(p-1)}{(m+p-1)^2}\dint_{\mathbb{R}^n}|\nabla\rho^{\frac{m+p-1}{2}}|^2 dx +\sigma_0\dint_{\mathbb{R}^n}|(-\Delta)^{\frac{2-n-B}{4}} \rho^{\frac{p+1}{2}}|^2\,dx\nonumber\\
 &+C\left(n,\frac{n+B}{2}\right)\frac{1}{B}\frac{4p(p-1)}{(p+1)^2}\dint_{\mathbb{R}^n}|(-\Delta)^{\frac{2-n-B}{4}} \rho^{\frac{p+1}{2}}|^2\,dx\nonumber\\
 &+(p-1)C(M_0,A,n)\dint_{\mathbb{R}^n}\rho^p\,dx+C(p,n,B,\sigma_0).
 \end{align*}
Taking
$$
\sigma_0=-C\left(n,\frac{n+B}{2}\right)\frac{1}{B}\frac{4p(p-1)}{(p+1)^2},
$$
we have
\begin{align*}
\frac{d}{dt}\dint_{\mathbb{R}^n} \rho^{p} \,dx
\leq&-\dfrac{4mp(p-1)}{(m+p-1)^2}\dint_{\mathbb{R}^n}|\nabla\rho^{\frac{m+p-1}{2}}|^2 dx\\
& +(p-1)C(M_0,A,n)\dint_{\mathbb{R}^n}\rho^p\,dx+C(p,n,B,\sigma_0).
 \end{align*}
Using the same procedure as (\ref{inteineq1})-(\ref{inteineq4}), we get
\begin{eqnarray*}
\frac{d}{dt}\dint_{\mathbb{R}^n} \rho^{p} \,dx
\leq -\|\rho\|^p_{L^p} +C(A,B,n,m,M_0,p).
\end{eqnarray*}
A simple computation shows that
\begin{eqnarray*}
\|\rho\|^p_{L^p}\leq \|\rho_0\|^p_{L^p}+C(A,B,n,m,M_0,p),
\end{eqnarray*}
for any fixed $p>1$, showing that the $L^p$-norm of $\rho$ is uniformly bounded in time.

For the case $A=2$, from (\ref{4inequal1nn}) and (\ref{4timeder2nn}) it can be obtained that
\begin{eqnarray*}
\frac{d}{dt}\dint_{\mathbb{R}^n} \rho^{p} \,dx
\leq -\dfrac{4mp(p-1)}{(m+p-1)^2}\dint_{\mathbb{R}^n}|\nabla\rho^{\frac{m+p-1}{2}}|^2 dx +n(p-1)M_0\dint_{\mathbb{R}^n}\rho^{p}\,dx,
\end{eqnarray*}
which implies the uniform in time $L^p$-bound for the case $A=2$.

For the case $A=2-n$, from (\ref{4caseAxiaodeng2}) and (\ref{4timeder1nn}) it can be obtained that
\begin{align*}
\frac{d}{dt}{\int_{\mathbb R^{n}}}\rho^{p}dx
\leq&-\frac{4mp(p-1)}{(m+p-1)^{2}}\int_{\mathbb R^{n}}|\nabla\rho^{\frac{m+p-1}{2}}|^{2}dx+(p-1)n\alpha(n)\int_{\mathbb R^{n}}\rho^{p+1}(x)dx\nonumber\\
&+C\left(n,\frac{n+B}{2}\right)\frac{1}{B}\frac{4p(p-1)}{(p+1)^2}\dint_{\mathbb{R}^n}|(-\Delta)^{\frac{2-n-B}{4}} \rho^{\frac{p+1}{2}}|^2\,dx\,.
\end{align*}
Utilizing (\ref{p+1}) and the Young inequality, we derive
\begin{eqnarray*}
\frac{d}{dt}\dint_{\mathbb{R}^n} \rho^{p} \,dx
\leq -\dfrac{4mp(p-1)}{(m+p-1)^2}\dint_{\mathbb{R}^n}|\nabla\rho^{\frac{m+p-1}{2}}|^2 dx +C(n,m,p,M_0,A,B),
\end{eqnarray*}
which implies the uniform in time $L^p$-bound for the case $A=2-n$.

\vskip 6pt

{\bf Step 2.-} In this step, we now prove the following iterative inequality
\begin{eqnarray}\label{3diedaistr}
\frac{d}{dt}\| \rho\|_{L^{p_k}}^{p_k}
\leq-\|\rho\|^{p_k}_{L^{p_k}} +\tilde C p_k^{\nu_2}\left(\left(\| \rho\|_{L^{p_{k-1}}}^{p_{k-1}}\right)^{\eta_1}+\left(\| \rho\|_{L^{p_{k-1}}}^{p_{k-1}}\right)^{\eta_2}\right),
\end{eqnarray}
where $0<\eta_1,\eta_2\leq 2$, $1<\nu_2\leq \max\{n+1,\frac{n}{2-n-B}+1\}$, $p_k=2^k+n$ ($k=0,1,2,...$), and $\tilde C$ is a constant independent of $p_k$.

For the case $2-n<A<2$, taking $p_k\rho^{p_k-1}$, $p_k=2^k+n>1$ as a test function in (\ref{model}) and using the same procedure to obtain (\ref{4timeder2nnn}), we have
\begin{align}\label{Lpstr}
\frac{d}{dt}\dint_{\mathbb{R}^n} \rho^{p_k} \,dx
\leq&-\dfrac{4mp_k(p_k-1)}{(m+p_k-1)^2}\dint_{\mathbb{R}^n}|\nabla\rho^{\frac{m+p_k-1}{2}}|^2 dx \nonumber\\ &+(p_k-1)n\alpha(n)\dint_{\mathbb{R}^n}\rho^{p_k+1}\,dx+(p_k-1)C(M_0,A,n)\dint_{\mathbb{R}^n}\rho^{p_k}\,dx\nonumber\\
&+C\left(n,\frac{n+B}{2}\right)\frac{1}{B}\frac{4p_k(p_k-1)}{(p_k+1)^2}\dint_{\mathbb{R}^n}|(-\Delta)^{\frac{2-n-B}{4}} \rho^{\frac{p_k+1}{2}}|^2\,dx\nonumber\\
\leq&-2C_1\dint_{\mathbb{R}^n}|\nabla\rho^{\frac{m+p_k-1}{2}}|^2 dx +C\left(n,\frac{n+B}{2}\right)\frac{1}{2B}\dint_{\mathbb{R}^n}|(-\Delta)^{\frac{2-n-B}{4}} \rho^{\frac{p_k+1}{2}}|^2\,dx\nonumber\\
&+C(M_0,A,n)(p_k-1)\dint_{\mathbb{R}^n}\rho^{p_k}\,dx +n\alpha(n)(p_k-1)\dint_{\mathbb{R}^n}\rho^{p_k+1}\,dx ,
\end{align}
where $0<C_1<\frac{2mp_k(p_k-1)}{(m+p_k-1)^2}$ is a constant independent of $p_k$.
Using again the interpolation inequality and the fractional Sobolev inequality \cite{CT}, we obtain
\begin{align*}
\|\rho\|^{p_k+1}_{p_k+1}&\leq\|\rho\|_{L^{p_{k-1}}}^{(p_k+1)\theta_3}\|\rho\|^{(p_k+1)(1-\theta_3)}_{L^{\frac{n(p_k+1)}{2n-2+B}}}\\
&\leq S\Big(n,\frac{2-n-B}{2}\Big)^{1-\theta_3}\|\rho\|_{L^{p_{k-1}}}^{(p_k+1)\theta_3}\|(-\Delta)^{\frac{2-n-B}{4}} \rho^{\frac{p_k+1}{2}}\|_{L^2}^{2(1-\theta_3)},
\end{align*}
where $\theta_3=\frac{\frac{1}{p_k+1}-\frac{2n-2+B}{n(p_k+1)}}{\frac{1}{p_{k-1}}-\frac{2n-2+B}{n(p_k+1)}}\sim O(1)$. Hence the Young inequality implies that for any $\sigma_3>0$
 \begin{align}\label{3inteineq1dstr}
n\alpha_n (p_k-1)\|\rho\|^{p_k+1}_{p_k+1}\leq&\, C(\sigma_3)\left(n\alpha_n S\Big(n,\frac{2-n-B}{2}\Big)^{1-\theta_3}(p_k-1)\|\rho\|_{L^{p_{k-1}}}^{(p_k+1)\theta_3}\right)^{q_1}\nonumber\\
&+\sigma_3 \left(\|(-\Delta)^{\frac{2-n-B}{4}} \rho^{\frac{p_k+1}{2}}\|_{L^2}^{2(1-\theta_3)}\right)^{q_2},
\end{align}
where $C(\sigma_3)=(\sigma_3 q_2)^{-\frac{q_1}{q_2}}q_1^{-1}$,
where $q_1, q_2>1$ satisfy $\tfrac{1}{q_1}+\tfrac{1}{q_2}=1$. Setting
\begin{eqnarray}\label{3sigmastr}
(1-\theta_3)q_2=1 \qquad \mbox{and} \qquad
\sigma_3=-C\left(n,\frac{n+B}{2}\right)\frac{1}{2B}\,,
\end{eqnarray}
we can compute
\begin{eqnarray}\label{q-1}
q_1=\frac{n(p_k+1)-(2n-2+B)p_{k-1}}{(2-B-n)p_{k-1}}<\frac{n}{2-n-B}+1.
\end{eqnarray}
 Therefore using (\ref{Lpstr}), (\ref{3inteineq1dstr}) and (\ref{3sigmastr}), we deduce
\begin{align*}
\frac{d}{dt}\dint_{\mathbb{R}^n} \rho^{p_k} \,dx
\leq&-2C_1\dint_{\mathbb{R}^n}|\nabla\rho^{\frac{m+p_k-1}{2}}|^2 dx +C(M_0,A,n)p_k\dint_{\mathbb{R}^n}\rho^{p_k}\,dx\nonumber \\
&+C(\sigma_3)\left(n\alpha_n S\Big(n,\frac{2-n-B}{2}\Big)^{1-\theta_3}(p_k-1)\|\rho\|_{L^{p_{k-1}}}^{(p_k+1)\theta_3}\right)^{q_1},
\end{align*}
Using the same process as (\ref{inteineq3d})-(\ref{timeder3d'}), we have
\begin{align}\label{Lpstr2}
\frac{d}{dt}\|\rho\|_{L^{p_k}}^{p_k}\leq& -\|\rho\|_{L^{p_k}}^{p_k}+C(\sigma_1)\left(1+(C(A,B,n,M_0)p_k)^{\ell_2}\right) S_n^{-\frac{\ell_2p_k(1-\theta_2)}{m+p_k-1}}\left(\|\rho\|^{p_{k-1}}_{L^{p_{k-1}}}\right)^{\frac{\ell_2 p_k\theta_2}{p_{k-1}}}\nonumber\\
&+C(\sigma_3)\left(n\alpha_n S\Big(n,\frac{2-n-B}{2}\Big)^{1-\theta_3}(p_k-1)\|\rho\|_{L^{p_{k-1}}}^{(p_k+1)\theta_3}\right)^{q_1}.
\end{align}
Notice (\ref{csigma}) and
\begin{align*}
C(\sigma_3)\left(n\alpha_n S\Big(n,\frac{2-n-B}{2}\Big)^{1-\theta_3}\right)^{q_1}\to C(m,n,A,B,M_0),~~\mbox{ as } p_k\to\infty,
\end{align*}
hence there is a constant $\tilde C:=\tilde C(m,n,A,B,M_0)>1$ independent of $p_k$ such that
$$
\max\left\{C(\sigma_1)\left(1+C(A,B,n,M_0)^{\ell_2}\right) S_n^{-\frac{\ell_2p_k(1-\theta_2)}{m+p_k-1}}\!\!\!\!\!,C(\sigma_3)\left(n\alpha_n S\Big(n,\frac{2-n-B}{2}\Big)^{1-\theta_3}\right)^{q_1}\right\}\leq \tilde C.
$$
Taking $1<\nu_2:=\max\{q_1,l_2\}\leq \max\{n+1,\frac{n}{2-n-B}+1\}$ due to (\ref{q-1}) and (\ref{ell}), and using the fact that
$$
\eta_1:=\frac{\ell_2 p_k\theta_2}{p_{k-1}}\leq 2 \mbox{ due to (\ref{ell})},\quad \eta_2:=\frac{p_k+1}{p_{k-1}}\theta_3q_1=\frac{p_k+1}{p_{k-1}}\leq 2,
$$
then (\ref{Lpstr2}) implies that (\ref{3diedaistr}) holds.

It is not difficult to see that the case $A=2-n$ is similar to the case $2-n<A<2$. The inequalities above hold true in the limit case $A=2-n$ with some simplications, and an iterative inequality can be also obtained.

Finally, Step 3 in the proof of Theorem \ref{th3.1b} shows the uniform $L^{\infty}$-bound in time for $2-n\leq A<2$.

The case $A=2$ follows similarly to the first case and the proof of Theorem \ref{th2.1}, an iterative inequality and the uniform $L^{\infty}$-bound of solutions can be analogously obtained.
\end{proof}

\section*{acknowledgements}
{\small JAC was partially supported by the EPSRC grant number EP/P031587/1.
JW is partially supported  by Program for Liaoning Excellent Talents in University (Grant No. LJQ2015041) and Key Project of Education Department of Liaoning Province (Grant No. LZD201701). The authors warmly thank the referees for the very valuable comments and suggestions that allowed us to improve this paper. We also thank Edoardo Mainini for pointing us out an important reference.}

\bibliographystyle{abbrv}
\bibliography{CW}

\end{document}